\documentclass{amsart}
\usepackage{amsfonts,amssymb,amsmath,amsthm}
\usepackage{url}
\usepackage{enumerate}
\usepackage{graphicx}

\newtheorem{thm}{Theorem}[section]
\newtheorem{lem}[thm]{Lemma}
\newtheorem{prop}[thm]{Proposition}
\newtheorem{cor}[thm]{Corollary}
\newtheorem{conj}[thm]{Conjecture}
\theoremstyle{definition}
\newtheorem{defn}[thm]{Definition}
\newtheorem{exm}[thm]{Example}
\newtheorem{rem}[thm]{Remark}
\numberwithin{equation}{section}

\author{P. Wright}
\address{
Mathematics Department\\
University of Western Australia\\
Western Australia}
\email{paul.e.wright@uwa.edu.au}
\thanks{}

\begin{document}

\title[Estimates of Hausdorff dimension]{Estimates of Hausdorff dimension for non-wandering sets of higher dimensional open billiards}

\begin{abstract}
This article concerns a class of open billiards consisting of a finite number of strictly convex, non-eclipsing obstacles $K$. The non-wandering set $M_0$ of the billiard ball map is a topological Cantor set and its Hausdorff dimension has been previously estimated for billiards in $\mathbb{R}^2$, using well-known techniques. We extend these estimates to billiards in $\mathbb{R}^n$, and make various refinements to the estimates. These refinements also allow improvements to other results. We also show that in many cases, the non-wandering set is confined to a particular subset of $\mathbb{R}^n$ formed by the convex hull of points determined by period 2 orbits. This allows more accurate bounds on the constants used in estimating Hausdorff dimension. 
\end{abstract}

\maketitle

\section{Introduction}
A billiard is a dynamical system in which a single pointlike particle moves at constant speed in some domain $Q \subset \mathbb{R}^D$ and reflects off the boundary $\partial Q$ according to the classical laws of optics \cite{ChaoticBilliards}. We describe a particle in the billiard by $x_t = (q_t, v_t)$ where $q_t \in Q$ is the position of the particle and $v_t \in \mathbb{S}^{D-1}$ is its velocity at time $t$. Then for as long as the particle stays inside $Q$, it satisfies 
$$(q_{t+s}, v_{t+s}) = S_s(x_t)= (q_t + s v_t, v_t).$$

\noindent Collisions with the boundary are described by 
$$v^+ = v^- - 2 \langle v^-, n \rangle n,$$
where $n$ is the normal vector (into $Q$) of $\partial Q$ at the point of collision, $v^-$ is the velocity before reflection and $v^+$ is the velocity after reflection.

Open billiards are a class of billiard in which the domain $Q$ is unbounded. We consider open billiards in which $Q = \mathbb{R}^D \backslash K$, where $K = K_1 \cup \ldots \cup K_u$ is a union of pairwise disjoint, compact and strictly convex sets with $C^2$ boundary, for some integer $u \geq 3$. The $K_i$ are called \textit{obstacles}. We assume that the \textit{no-eclipse condition} $(\textbf{H})$ holds. That is, for any nonequal $i,j,k$, the convex hull of $K_i \cup K_j$ does not intersect $K_k$. This condition ensures that the non-wandering set (defined later) does not include trajectories that are tangent to the boundary.

We denote by $n = n_K(q)$ the outward normal vector of $\partial K$ at $q$. Let $\hat{Q} = \{(q,v) \in Q \times \mathbb{S}^{D-1} | q \in \mbox{int } Q \mbox{ or } \langle n,v \rangle \geq 0 \}$ be the phase space of $S_t$ with canonical projection $\pi: \hat{Q} \rightarrow Q$. Let $M = \{(q,v) \in \partial K \times \mathbb{S}^{D-1} | \langle n,v \rangle \geq 0 \}$ be the boundary of $\hat{Q}$.

The \textit{non-wandering set} of a dynamical system is the set of points whose trajectories never escape from the system i.e. the set of points $x$ such that the full trajectory $\{S_t(x): t \in \mathbb{R}\}$ is bounded. The non-wandering set of the flow is denoted $\Omega(S)$ or $\Omega$. Its restriction to the boundary is $M_0 = \Omega \cap (\partial K \times S^{D-1})$. Equivalently, $M_0 = \{x \in M: |t_j(x)| < \infty \mbox{ for all } j \in \mathbb{Z} \}$, where $t_j(x) \in [-\infty, \infty]$ denotes the time of the $j$-th reflection of $x \in \hat{Q}$. Let $d_j(x) = t_j(x) - t_{j-1}(x)$. Let $\hat{Q}' = t_1^{-1}(0,\infty)$, $M' = M \cap \hat{Q}'$ and define the \textit{billiard ball map} as $B: M' \rightarrow M, x \rightarrow S_{t_1(x)}(x)$. Then $B$ is invertible and $C^2$ (in general $B$ is at least as smooth as the boundaries of the obstacles), except where $v$ is tangent to $K$ at $Bx$, and its restriction to $M_0$ is a bijection. $M_0$ is the non-wandering set of the billiard ball map; this non-wandering set is the main focus of this paper.

\section{Main Theorem}

The main result of this paper is in three parts.

\begin{thm}

Let $K = K_1 \cup \ldots \cup K_u \subset \mathbb{R}^D$ be disjoint, compact and strictly convex sets with smooth boundary, for some integer $u \geq 3$. Let $B$ be the \textit{billiard ball map} in $Q = \mathbb{R}^D \backslash K$. Let $\lambda_1^{-1} = 1 + d_{\max} g_{\max}$ and $\mu_1^{-1} = 1 + d_{\min} g_{\min}$, where $d_{\min}$, $d_{\max}$, $g_{\min}$ and $g_{\max}$ are constants depending on the billiard, defined in Sections \ref{Properties of open billiards} and \ref{better est}. Then the Hausdorff dimension of the non-wandering set $M_0$ of $B$ is given as follows:

\begin{enumerate}
	\item If $D = 2$, then
	\begin{equation} \label{dim est}
	\frac{-2 \ln (u-1)}{\ln \lambda_1} \leq \dim_H M_0 \leq \frac{-2 \ln (u-1)}{\ln \mu_1}.
	\end{equation}
	
	\item If $D \geq 3$, and the obstacles $K_i$ are sufficiently far apart that $\lambda_1^{d_{\max}} < \mu_1^{2 d_{\min}}$, then equation (\ref{dim est}) holds, although note that $g_{\min}$ is different in the higher dimensional case.
	
	\item We always have
	\begin{equation} \label{dim est alpha}
	\alpha \frac{-2 \ln (u-1)}{\ln \lambda_1} \leq \dim_H M_0 \leq \alpha^{-1} \frac{-2 \ln (u-1)}{\ln \mu_1},
	\end{equation}
	where $\alpha = \frac{2 d_{\min} \ln \mu_1}{d_{\max} \ln \lambda_1}$ is a particular H\"{o}lder constant, calculated in section \ref{Holder const}.

\end{enumerate}

\end{thm}

\begin{rem}
Hassleblatt and Schmeling present a conjecture in \cite{HSconjecture} that would imply that $\alpha = 1$ for any billiard, making the above theorem much stronger. This will be discussed in section \ref{DPS}.
\end{rem}

Part 1 was essentially proved in \cite{Kenny}, except that the improvements to estimates in Section \ref{better est} can be applied. We deal with the higher dimensional case here.

\section{Properties of open billiards}
\label{Properties of open billiards}

The following lemma is well known (see for example \cite{Stoyanov2}).
\begin{lem}
\label{findorbits}

If $K$ satisfies the no-eclipse condition $(\textbf{H})$, then for any finite sequence of indices $1 \leq i_1, \ldots , i_n \leq u$ $(n \geq 3)$ such that $i_j \neq i_{j+1}$ for all $j$, let

$$F: K_{i_1} \times \ldots \times K_{i_n} \rightarrow \mathbb{R}, (q_1, \ldots, q_n) \mapsto \displaystyle \sum_{j=1}^n \|q_j - q_{j+1}\|,$$

\noindent where we denote $q_{n+1} = q_1$. Then $F$ achieves its minimum at some $(p_1, \ldots p_n)$ such that $p_j \in \partial K_{i_j}$ for all $j$. Specifically, the $p_j$ are the successive reflection points of a periodic billiard trajectory in $Q$ with $p_{j+1} = B p_j$ and $p_1 = B p_n$.

\end{lem}

\subsection{Billiard constants}

\begin{defn}
At each point on a hypersurface $M$, the shape operator or second fundamental form (s.f.f.) $S_p : T_p(M) \rightarrow T_p(M)$ is defined by $S_p(v) = - \nabla_v n_M(p)$. The curvature of M at $p$ in the direction of a unit vector $\hat{u} \in T_p(M)$ is $k_p(\hat{u}) = S_p(\hat{u}) \cdot \hat{u}$.
\end{defn}

Every billiard has several associated constants that can be useful in various estimates. The s.f.f $S_q(K)$ of $K$ at $q$ has $n-1$ eigenvalues, or principle curvatures. Let $\kappa_{\min}(q), \kappa_{\max}(q)$ denote the smallest and largest eigenvalues respectively at $q$. The billiard has minimum and maximum curvatures $\kappa^- = \displaystyle \min_{q \in \pi M_0} \kappa_{\min}(x)$ and $\kappa^+ = \displaystyle \max_{x \in M_0} \kappa_{\max}(x)$. We denote $d_{\min} = \min\{ d^-_{ij} : 1 \leq i,j \leq u \}$ and $d_{\max} = \max \{d^-_{ij} : 1 \leq i,j \leq u \}$, where $d^-_{ij}$ and $d^+_{ij}$ are the respective minimum and maximum of the set $\{d(\pi x, \pi y): x \in K_i \cap M_0, y \in K_j \cap M_0 \}$. For a point $x = (q,v) \in M$, we call $\phi(x) = \arccos \langle v, \nu_K(q) \rangle$ the \textit{collision angle}, the acute angle which the $j$-th reflected ray makes with the outer normal to $K$. We denote $\phi_j(x) = \phi(B^j x)$. The collision angle can be bounded above by some constant $\phi^+ = \max\{\phi(x): x \in M_0 \}$. It can easily be shown that $\phi^+ \leq \arccos(b^- / d_{\max})$, where $b^- = \displaystyle \min_{i,j,k} d(K_j, \mbox{Cvx}(K_i,K_k))$.


\section{Convex fronts}

Let $X$ be a smooth, stricly convex $D-1$ dimensional surface in int $Q$ with outer normal field $v(q)$, let $\hat{X} = \{(q, v(q)): q \in X \}$, $\hat{X}_0 = \hat{X} \cap \Omega$ and $X_0 = \pi \hat{X_0}$ where $\pi$ is the canonical projection.  Let $x,y \in X$. Let $Y: q(s), s \in [0,1]$ be a $C^3$ curve on $X$ with outer normal field parametrised by $v(s) = v(q(s))$. Let $Y_0 = Y \cap X_0$, $\hat{X}_t = S_t(\hat{X})$, $X_t = \pi \hat{X}_t$, $\hat{Y}_t = S_t(\hat{Y})$, $Y_t = \pi \hat{Y}_t$ and $t_j(s) = t_j(q(s),n(s))$. Where defined, let $q_j(s) = \pi B^j(q(s),v(s))$ be the $j$-th reflection point of $(q(s),v(s))$, then let $d_j(s) = t_j(s) - t_{j-1}(s)$, and $\phi_j(s) = \phi_j(q(s),v(s))$.

For a point $q \in \partial K$, let $\mathcal{J}$ denote the tangent space $T_q(X)$ of the convex front, and let $\mathcal{T}$ denote the tangent space of $\partial K$ at $q$. The s.f.f of $X$ at $q$ is given by $\mathcal{B}: \mathcal{J} \rightarrow \mathcal{J}$, $\mathcal{B} dv = S_q(dv)$.

\subsection{Evolution of Fronts}

With no collisions, the curvature of a convex front $X$ front is given by the formula \cite{BalChe} 
$$\mathcal{B}(q_t(s)) = (\mathcal{B}(q)^{-1} + t I)^{-1}.$$

At a collision point, let $B^-$ be the second fundamental form just before the collision and let $\mathcal{B}^+$ be the s.f.f. just after the collision. Then 
$$\mathcal{B}^+ = \mathcal{B}^- + 2 \Theta = \mathcal{B}^- + 2 \langle n, v \rangle V^* KV,$$
where $V: \mathcal{J} \rightarrow \mathcal{T}$ is the projection $V dv = dv - \frac{\langle dv, n \rangle}{\langle n,v \rangle} v \in \mathcal{T}$, $K: \mathcal{T} \rightarrow \mathcal{T}$ is the s.f.f. of $K$ at $q$, $V^*: \mathcal{T} \rightarrow \mathcal{J}$ is the projection $V^* dq = dq - \frac{\langle dq, v \rangle}{\langle n,v \rangle} n \in \mathcal{J}$, and $\langle n, v \rangle = \cos \phi$ where $\phi \in [0,\frac{\pi}{2}]$ is the collision angle. 

\subsection{Estimating $\Theta$}
\label{Estimating Theta}
\begin{lem}
If the dimension $n$ is greater than 2, let $\kappa_{\min}$, $\kappa_{\max}$ be the smallest and largest eigenvalues of the s.f.f. $K$ at $q$, so that $\kappa_{\min} |dq| \leq \|Kdq \| \leq \kappa_{\max} |dq|$. Then
$$\kappa_{\min} \cos \phi \leq \|\Theta\| \leq \frac{\kappa_{\max}}{\cos \phi}.$$
\end{lem}

\begin{proof}
If $n = v$ then $\langle n, v \rangle = \cos \phi = 1$ so $\Theta = \langle n,v \rangle V^*KV = K$ and the inequality holds. Henceforth we assume $n \neq v$. Let $\mathcal{S} = \mathcal{J} \cap \mathcal{T}$. Any vector $dv \in \mathcal{J}$ can be written in the form $dv = |dv| (\hat{a} \cos{\theta} + \hat{s} \sin{\theta})$, where $\hat{s} \in \mathcal{S}$ and $\hat{a} \in \mathcal{J}$ are unit vectors, $\hat{a}$ is perpendicular to $\mathcal{S}$, and $\langle \hat{a}, n \rangle \geq 0$. Then $\hat{a}$ is in the plane containing by $n$ and $v$ so the angle between $\hat{a}$ and $n$ is $\frac{\pi}{2} - \phi$. Using $dv \perp v$ we get

\begin{align*}
\|V dv\|  &= \|dv - \frac{\langle |dv| \hat{s} \sin{\theta}, n \rangle}{\langle n,v \rangle} v - \frac{\langle |dv| \hat{a} \cos{\theta}, n \rangle}{\langle n,v \rangle} v \|\\
					&= \|dv - \left( |dv| \tan \phi \cos \theta \right) v \|\\
					&= \sqrt{1 + \tan^2 \phi \cos^2 \theta} |dv|
\end{align*}

Similarly, write $dq \in \mathcal{T}$ as $dq = |dq| ((\hat{b} \cos{\theta'} + \hat{s'} \sin{\theta'})$ for some unit vectors $\hat{s'} \in \mathcal{S}$ and $\hat{b} \in \mathcal{T}$ with $\hat{b} \perp \mathcal{S}$ and $\langle \hat{b}, \hat{v} \rangle \geq 0$. Then $\|V^* dq\| = \sqrt{1 + \tan^2 \phi \cos^2 \theta'} |dq|$. Combining these operator norms and using $0 \leq \cos^2 \theta, \cos^2 \theta' \leq 1$, we get
\begin{align*}
\kappa_{\min} \cos \phi &\leq \cos \phi \sqrt{1 + \tan^2 \phi \cos^2 \theta} \kappa_{\min} \sqrt{1 + \tan^2 \phi \cos^2 \theta'} \\
						&\leq \|\Theta\|	\leq \cos \phi \sqrt{1 + \tan^2 \phi \cos^2 \theta} \kappa_{\max} \sqrt{1 + \tan^2 \phi \cos^2 \theta'}\\
						&\leq \cos \phi \kappa_{\max} (1 + \tan^2 \phi) \leq \frac{\kappa_{\max}}{\cos \phi}
\end{align*}
as required.
\end{proof}

Note that in the two dimensional case, $\theta = \theta' = 0$ since $\mathcal{S} = \mathcal{T} \cap \mathcal{J} = \{q\}$, and $\kappa_{\min} = \kappa_{\max} = \kappa$ at every point. So the inequality becomes $\|\Theta\| = \frac{\kappa}{\cos \phi}$.

\subsection{Estimating $k_j$}
\label{deltature}

This section follows the definitions in \cite{Non-integrability}. Let $u_j(s) = \displaystyle \lim_{\tau \downarrow t_j(s)} \frac{d}{ds} S_{\tau} q(s)$ and let $\hat{u}_j(s) = \frac{u_j(s)}{\|u_j(s)\|}$ be the unit tangent vector of $Y_t$ at $q(s)$. Let $\mathcal{B}_j$ be the s.f.f. of $S_{t_j(s)} X$ at $q_j(s)$. Define $\ell_j(s) > 0$ by

$$[1 + d_j(s) \ell_j(s)]^2 = \| \hat{u}_j(s) + d_j(s) \mathcal{B}_j \hat{u}_j(s) \|^2,$$

Then set $\delta_j(s) = \frac{1}{1 + d_j(s) \ell_j(s)}.$

\begin{prop} 
\label{q = p delta}
Fix a point $x_0 = (q_0,v_0) \in \hat{X}$, a positive integer $m$ and some $\tau$ with $t_m(x_0) < \tau < t_{m+1}(x_0)$. Let $Y: [0,a] \rightarrow X$ be a $C^3$ curve with $q(0) = q_0$ with $a$ small enough that for every $s \in [0,a]$ we have $t_m(x(s)) < \tau < t_{m+1}(x(s))$, where $x(s) = (q(s), \nu_X(q(s)))$, and that for all $j = 1, \ldots , m$ the points $q_j(s) \in \partial K_{i_j}$ for all $s \in [0,a]$. Then $p(s) = \pi S_t(x(s))$ is a $C^3$ curve on $X_t$. For all $s \in [0,1]$ we have

$$\|q'(s) \| = \frac{\|p'(s) \|}{1 + (\tau - t_m(s)) k_m(s))} \delta_0(s) \delta_1(s) \ldots \delta_m(s).$$

\begin{proof}
See \cite{Non-integrability, Spectrum}. The same result can be derived from \cite{BalChe}, and is also proved for completeness (in two dimensions only) in \cite{Kenny}.
\end{proof}

\end{prop}

Now the curvature of the convex front after $j$ reflections in the direction $\hat{u}_j$ is $k_j = \langle \mathcal{B}_j \hat{u}_j, \hat{u}_j \rangle$, so

$$1 / \delta_j(s)^2 = 1 + 2 d_j(s) k_j(s) + d_j(s)^2 \|\mathcal{B}_j \hat{u}_j(s) \|^2.$$

Let $q \in X$ and let $x = (q, \nu_X(q))$. Let $\mu_j(s)$ and $\lambda_j(s)$ be the minimum and maximum eigenvalues of $\mathcal{B}_j(q(s))$ respectively.

Recall that $\mathcal{B}_{j+1} = \mathcal{B}_{j+1}^- + 2 \Theta = (\mathcal{B}_j^{-1} + d_j I)^{-1} + 2 \Theta$. $\mathcal{B}_j$ is always positive definite, so $\mu_j$ and $\lambda_j$ are always positive. Note that if $\lambda$ is an eigenvalue of $\mathcal{B}(q(s))$, then $\frac{\lambda}{1 + t \lambda}$ is an eigenvalue of $\mathcal{B}(q_t(s))$. So we have $\lambda_{j+1} = \frac{\lambda_j}{1 + d_j \lambda_j} + \frac{2 \kappa_{\max}(x_j)}{\cos \phi_j(x)}$ and $\mu_{j+1} = \frac{\mu_j}{1 + d_j \mu_j} + 2 \kappa_{\max}(x_j) \cos \phi_j(x)$. For all $j \geq 0$, $\mu_j(s) \leq k_j(s) \leq \lambda_j(s)$, so we get $k_{j+1}(s) \in$
\begin{equation}
\label{curvature estimate}
\left[ \frac{k_j(s)}{1 + d_j(s) k_j(s)} + 2 \kappa_{\min}(x_j(s)) \cos \phi_j(s), \frac{k_j(s)}{1 + d_j(s) k_j(s)} + \frac{2 \kappa_{\max}(x_j(s))} {\cos \phi_j(s)} \right].
\end{equation}

\section{Coding $M_0$ and $X_0$}

For each $x \in M_0$ we have a bi-infinite sequence of indices $\alpha = \{\alpha_i\}_{i=-\infty}^\infty$, $\alpha_i \in \{1, \ldots , u\}$ such that $\pi B^i x \in \partial K_{\alpha_i}$. Since each $K_i$ is convex, $\alpha_i \neq \alpha_{i+1}$ for all $i$, so define the symbol spaces $\Sigma$ and $\Sigma^+$ as

$$\Sigma = \left \{ (\alpha_i)_{i = -\infty}^\infty : \alpha_i \in \{1, \ldots, u\}, \alpha_i \neq \alpha_{i+1} \mbox{ for all } i \in \mathbb{Z} \right \},$$
$$\Sigma^+ = \left \{ (\alpha_i)_{i = 1}^\infty : \alpha_i \in \{1, \ldots, u\}, \alpha_i \neq \alpha_{i+1} \mbox{ for all } i \leq 0 \right \}.$$

Let $f: M_0 \rightarrow \Sigma, x \mapsto \alpha$ denote the representation map. The two-sided subshift $\sigma: \Sigma \rightarrow \Sigma, \alpha_i \mapsto \alpha_{i+1}$ is continuous under the following metric $d_\theta$ for any $\theta \in (0,1)$.

\begin{displaymath}
d_\theta(\alpha,\beta) = \begin{cases}
       0:        &  \text{if $\alpha_i = \beta_i$ for all $i \in \mathbb{Z}$}  \\
       \theta^n: &  \text{if $n = \max \{j \geq 0: \alpha_i = \beta_i \mbox{ for all } |i| < j\}$},
\end{cases}
\end{displaymath}

We define a similar metric on $\Sigma^+$. 

\begin{displaymath}
d_\theta(\alpha,\beta) = \begin{cases}
       0:        &  \text{if $\alpha_i = \beta_i$ for all $i \geq 0$}  \\
       \theta^n: &  \text{if $n = \max \{j \geq 0: \alpha_i = \beta_i \mbox{ for all } 0 \leq i \leq j\}$},
\end{cases}
\end{displaymath}

\begin{lem}
If $u \geq 2$ and $\theta \in (0,1)$, then $f$ is a homeomorphism of $M_0$ (with the topology induced by $M$) onto $(\Sigma, d_\theta)$, and the shift $\sigma$ is topologically conjugate to $B$, that is $B = f^{-1} \circ \sigma \circ f$.

\end{lem}

Assuming $u \geq 3$, $M_0$ is a compact topological Cantor set. $B$ is topologically transitive on $M_0$ and its periodic points are dense in $M_0$. $B$ is hyperbolic on $M_0$, and $M_0$ is a basic set for $B$.

Given the surface $X$, the intersection $X_0 = X \cap \Omega$ can also be coded by sequences. Define the representation map $\Upsilon : X_0 \rightarrow \Sigma^+$ in the same way as $f : M_0 \rightarrow \Sigma$. Define an equivalence relation $\sim_m$ $(m \geq 0)$ by $\alpha \sim_m \beta \Leftrightarrow \alpha_i = \beta_i$ for all $1 \leq i \leq m$, and $\alpha \sim_0 \beta$ for any $\alpha, \beta \in \Sigma^+$. We call the equivalence classes $[\alpha]_m$ \textit{cylinders}. Define another relation (not an equivalence relation) $\approx_m$ by $\alpha \approx_m \beta$ if $\alpha \sim_m \beta$ and $\alpha_{m+1} \neq \beta_{m+1}$.

The following lemma on Hausdorff dimension and packing dimension is the result of direct calculations (see for example \cite{Edgar, Kenny}).

\begin{lem} \label{packing dim lemma}
For any $\alpha \in \Sigma^+$ and $N \in \mathbb{N}$,
$$\overline{\dim_p}([\alpha]_N, d_\theta) = \dim_H([\alpha]_N, d_\theta) = \frac{-\ln(u-1)}{\ln \theta}.$$

\end{lem}

We find upper and lower bounds $g_{\min}$ and $g_{\max}$ such that for some $N \in \mathbb{N}$, $k_j(s) \in [g_{\min},g_{\max}]$ for all $j \geq N$.

\section{Estimating $\delta_j(s)$}

Section 4.1 of \cite{Kenny} contains a significant improvement to the dimension estimate using the continued fraction for $k_j(s)$. We can do the same using the bounds in (\ref{curvature estimate}).

The map $f_{\gamma, \theta}: (0, \infty) \rightarrow \mathbb{R}, x \mapsto \frac{x}{1+ \theta x} + 2 \gamma$ has one positive fixed point $g(\gamma, \theta) = \gamma + \sqrt{\gamma^2 + 2 \gamma / \theta}$. This function is non-decreasing in $\gamma$ and strictly decreasing in $\theta$. 

The natural domain for $g$ is $[\kappa_{\min} \cos \phi^+, \frac{\kappa_{\max}}{\cos \phi^+}] \times [d_{\min}, d_{\max}]$ for the arguments of $g$. On this domain, the minimum and maximum values of $g$ are $g(\kappa_{\min}, d_{\max})$ and $g(\frac{\kappa_{\max}}{\cos \phi^+}, d_{\min})$ respectively. While this domain is an obvious choice, it is not the strictest or most useful domain. We will use a smaller domain $\mathbb{D}$ defined in Section \ref{better est}.

We write $g_{\min} = \displaystyle\max_{(\gamma, \theta) \in \mathbb{D}} g(\gamma, \theta)$ and $g_{\max} = \displaystyle \max_{(\gamma, \theta) \in \mathbb{D}} g(\gamma, \theta)$. The values that maximise and minimise $g$ are denoted $(\gamma_{\max}, \theta_{\min})$ and $(\gamma_{\min}, \theta_{\max})$ respectively. 

Parametrise the surface $X$ by $q(t) = q(t_1,\ldots, t_{D-1})$ where each $t_i \in [0,1]$ and $D$ is the dimension of the billiard. Let $UT(X) = \{(q,\hat{u}): q \in X, \|\hat{u}\| = 1, \hat{u} \mbox{ tangent to $X$ at $q$} \}$ denote the unit tangent bundle of $X$, and parametrise $UT(X)$ by $x(s) = x(t, \hat{u})$, where $s \in [0,1]^{D-1} \times \mathbb{S}^{D-2}$. Consider any $s = (t, \hat{u}) \in S$ such that $q(t) \in X_0$ and any sequences $(\gamma_j, \theta_j)_1^\infty \subset \mathbb{D}$. Let $k_0(s) = \mathcal{B}_0(t)(\hat{u}) \cdot \hat{u}$ be the curvature of $X$ at $q(t)$ in the direction $\hat{u}$, and inductively define $k_{j+1}(s) = f_{\gamma_j, \theta_j} (k_j(s))$ for $0 \leq j \leq n-1$.  

\begin{lem}
Let $a < g_{\min}$ and $b > g_{\max}$. Then there exists $n(X) > 0$ such that for all $s$ and $j \geq n(X)$ we have $k_j(s) \in [a,b]$. 

\begin{proof}
If $k_N(s) \leq g_{\max}$ for some $s$ and some $N \geq 0$ then inductively
$$k_{j+1}(s) = f_{\gamma_j, \theta_j}(k_j(s)) \leq f_{\gamma_{\max}, \theta_{\min}}(k_j(s)) \leq f_{\gamma_{\max}, \theta_{\min}}(g_{\max}) = g_{\max}$$
for all $j \geq N$. Similarly if $k_N(s) \geq g_{\min}$ for some $N$ then $k_j(s) \geq g_{\min}$ for all $j \geq N$. For each $s$, define $k^-_j$ and $k^+_j$ by 
$k^-_0 = k_0, k^-_{j+1} = f_{\gamma_{\min}, \theta_{\max}}(k^-_j)$ and 
$k^+_0 = k_0, k^+_{j+1} = f_{\gamma_{\max}, \theta_{\min}}(k^+_j)$. 
Then for all $j \geq 0$ and $s\in S$ we have $k^-_j(s) \leq k_j(s) \leq k^+_j(s)$, $\displaystyle \lim_{j \rightarrow \infty} k^-_j(s) = g_{\min}$ and $\displaystyle \lim_{j \rightarrow \infty} k^+_j(s) = g_{\max}$. There must be some integer $j_0(s) \geq 0$ such that $k_j(s) \in [a,b]$ for all $j \geq j_0(s)$.

Since $TX$ is compact, $k_0(s)$ has an infimum $k_{0,\min} = k_0(s_{\min})$ and a supremum $k_{0,\max} = k_0(s_{\max})$. Let $n(X) = \max\{j_0(s_{\min}), j_0(s_{\max}) \}$. Then for $j \geq n(X)$,

$$a \leq k_j(s_{\min}) \leq	k_j(s) \leq k_j(s_{\max}) \leq b,$$
so $j_0(s) \leq n(X)$ for all $s \in S$. Thus we have $k_j(s) \in [a,b]$ for all $j \geq n(X)$ as required.
\end{proof}
\end{lem}

For any $\tau \geq 0$, $n(X) \geq n(S_{\tau} X)$. So by taking a finite number of convex fronts $X_i$ whose image under $S_{\tau}$ covers $\Omega$, we can get a global constant $n_0 = n(a,b) = \max\{n(X_i) : M \subset \bigcup_i X_i \}$ that depends only on $a,b$ and the billiard itself.

Now $k_j(s) \in (a,b)$ for all $s \in q^{-1}(X)$ and $j > n_0$. So for these values,

$$\delta_j(s) \in \left(\frac{1}{1 + d_{\max} b}, \frac{1}{1 + d_{\min} a} \right).$$

Define $\lambda = \frac{1}{1 + d_{\max} b}$ and $\mu = \frac{1}{1 + d_{\min} a}$ for now. For $0 \leq j < n_0$, we can still find bounds for $\delta_j(s)$. $k_j(s)$ is always bounded below by $0$, and we can assume $k_0(s)$ is bounded above by some $k_0^+$ \cite{Sinai}. So $\delta_j(s) \in [\delta^-, 1]$ where $\delta^- = \frac{1}{1 + d_{\max} k_0^+}$. Furthermore, we have $2 \kappa^- \cos \phi^+ \leq k_j(s) \leq \frac{1}{d_{\min}} + \frac{2 \kappa^+}{\cos \phi^+}$ for $1 \leq j < n_0$. Thus, $\delta_j(s) \in [\lambda_0, \mu_0]$ where $\lambda_0^{-1} = 1 + d_{\max}(\frac{1}{d_{\min}} + \frac{2 \kappa^+}{\cos \phi^+})$ and $\mu_0^{-1} = 1 + 2 d_{\min} \kappa^- \cos \phi^+$.

\section{Hausdorff dimension of $X_0$}

\begin{prop}
Let $[a,b] \supset [g_{\min},g_{\max}]$, $\lambda = \frac{1}{1 + d_{\max} b}, \mu = \frac{1}{1 + d_{\min} a}$, and $n_0 = n(a,b)$ as defined above. There exist constants $c,C$ depending only on the billiard, such that for any integer $n \geq n_0$ and $x_1, x_2 \in \hat{X}_0$ such that $x_1 \approx_n x_2$, we have

$$c \lambda^{n-n_0} \leq \|\pi x_1 - \pi x_2 \| \leq C \mu^{n - n_0}.$$

\begin{proof}
Let $n \geq n_0$ and let $x_1, x_2 \in \hat{X}_0$ with $x_1 \approx_n x_2$. Without loss of generality assume $t_n(x_1) < t_n(x_2)$ and let $\tau = t_n(x_2)$. Let $y_1 = S_\tau x_1$, $y_2 = S_\tau x_2$. Now let $p(s)$ parametrize (by arc length) the shortest curve $\Gamma \subset S_\tau X$ between $y_1$ and $y_2$. Let $q(s) = S_{-\tau}(p(s))$ paramatrize the curve $Y = S_{-\tau} \Gamma$. This curve will not be the shortest curve between its endpoints $x_1$ and $x_2$, in fact for large $n$ it can be much longer. We have 

\begin{align*}
\|\pi x_1 - \pi x_2\| &= \left\| \int_Y q'(s) ds \right\| \leq \int_Y \|q'(s)\| ds\\
							&= \int_{\Gamma} \frac{\|p'(s)\|}{1 + (\tau - t_n(s)) k_n(s)} \left( \displaystyle\prod_{j=0}^{n-1} \delta_j(s)\right) ds \\
							&\leq \mu^{n-n_0} \mu_0^{n_0} \int_{\Gamma} ds  \leq C \mu^{n-n_0}.
\end{align*}

Here we used Proposition \ref{q = p delta}, $(\tau - t_n(s)) k_n(s) \geq 0$, $\delta_j(s) < \mu_0$ for $0 \leq j \leq n_0$, $\delta_j(s) < \mu$ for $j > n_0$. Since the curve $\Gamma$ is the shortest curve between two points on a surface with bounded curvature \cite{Sinai}, and confined to a bounded set (e.g. a ball containing $K$), its arc length $\int_{\Gamma} ds$ can be bounded above by a constant.

Now we find an estimate for $\|x_1 - x_2\|$ from below, using different curves. Let $q(s)$ parametrise the shortest curve $Y$ in $X$ between $x_1$ and $x_2$. Now let $[s_1,s_2] \subseteq [0,1]$ such that $s = s_1, s_2$ are the only values for which $(q(s), n(s))$ has an $(n+1)$-st reflection. Let $y_1 = q_{n+1}(s_1)$, $y_2 = q_{n+1}(s_2)$. Without loss of generality assume $t_{n+1}(s_1) < t_{n+1}(s_2)$ and let $\tau = t_{n+1}(s_1)$, $z = S_\tau(q(s_2))$. Then $p(s) = S_\tau q(s)$ parametrizes the curve $S_\tau \hat{Y}$.

We have constants $C_1$ and $C_2$ such that
\begin{align*}
\|\pi x_1 - \pi x_2\| & \geq C_1 \int_X \|q'(s)\| ds \geq C_1 \int_{s_1}^{s_2} \|q'(s)\| ds \\
							& = C_1 \int_{s_1}^{s_2} \frac{\|p'(s)\|}{1 + (\tau - t_n(s)) k_n(s)} \left( \displaystyle\prod_{j=0}^{n-1} \delta_j(s)\right)\\
							& \geq C_1 C_2 \lambda_0^{n_0} \lambda^{n-n_0}  \int_{s_1}^{s_2} \|p'(s)\| ds
\end{align*}

Clearly $z$ is in the convex hull of the two obstacles containing $q_n(s_2)$ and $y_2$ respectively, and $y_1$ is in a third obstacle. Thus we have $\int_{s_1}^{s_2} \|p'(s)\| \geq \|y_1 - z\| \geq b^-$, where $b^-$ is the minimum distance between $K_k$ and Cvx $(K_i \cup K_j)$ for any nonequal $i,j,k$. Letting $c = C_1 C_2 \lambda_0^{n_0} b^-$, we have $c \lambda^{n-n_0} \leq \|\pi x - \pi y \| \leq C \mu^{n - n_0}$ as required.

\end{proof}

\end{prop}

\begin{prop}

Let $0 < n_0 \leq n$. Suppose there are constants $c,C > 0$ such that $c \lambda^{n - n_0} \leq \| \pi x - \pi y \| \leq C \mu^{n-n_0}$ whenever $x,y \in \hat{Y_0}$ with $x \approx_n y$. Then $\Upsilon: \hat{Y}_0 \rightarrow \Sigma^+$ is injective and a Lipschitz homeomorphism from $\hat{Y_0}$ to the metric space $(\Upsilon(\hat{Y}_0), d_\lambda)$, and $\Upsilon^{-1}$ is a Lipschitz homeomorphism from $(\Upsilon(\hat{Y_0}), d_\mu)$ onto $\hat{Y_0}$.

\begin{proof}
For any $x \in X_0$ with sufficiently large $n \geq n_0$, there is some $z \in X_0$ such that $z \approx_n x$, so if $\Upsilon(x) = \Upsilon(y)$ then $\|x - y\| \leq \|x - z\| + \|y -z\| \leq 2 C \mu^n \rightarrow 0$ as $n \rightarrow \infty$. So $\Upsilon^{-1}$ is well defined and $\Upsilon$ is injective.

Let $x \approx_n y \in X_0$. Then $d_{\lambda}(\Upsilon x, \Upsilon y) = \lambda^n \leq \frac{1}{c} \|x - y\|$, so $\Upsilon$ is Lipschitz.

Similarly, for distinct $\alpha, \beta \in \Upsilon(X_0)$, $x \in \Upsilon^{-1}(\alpha)$, $y \in \Upsilon^{-1}(\beta)$, and $n$ such that $x \approx_n y \in X_0$, we have $\|\Upsilon^{-1}(\alpha) - \Upsilon^{-1}(\beta)\| \leq C \mu^n = C d_\mu(\alpha, \beta)$. Finally, since the identity $I: (\Upsilon(X_0),d_{\lambda}) \rightarrow (\Upsilon(X_0), d_{\mu})$ is continuous, the maps $\Upsilon: X_0 \rightarrow (\Upsilon(X_0), d_{\mu})$ and $\Upsilon^{-1}:(\Upsilon(X_0), d_{\lambda}) \rightarrow X_0$ are also continuous.

\end{proof}

\end{prop}

The following theorem is well known (see \cite{FractalGeometry})
\begin{thm}
Let $f: A \rightarrow B$ be a Lipschitz map and let $F \subset A$. Then $\dim_H f(F) \leq \dim_H F$.
\end{thm}

For some $\alpha \in \Sigma^+$ and sufficiently large $n \geq n_0$ the cylinder $[\alpha]_n \subset \Upsilon(\hat{Y}_0)$.
It follows that $\dim_H (\Upsilon(\hat{Y}_0), d_\lambda) \leq \dim_H \hat{Y}_0 \leq \dim_H(\Upsilon(\hat{Y}_0), d_\mu)$.

\section{Hausdorff dimension of $M_0$}

We now relate $\dim_H X_0$ to $\dim_H M_0$. Let $x \in M_0$ and let $\hat{X} = S_{\tau}(W_\theta^{(u)}(x))$ be the image of the local unstable manifold $W_\theta^{(u)}(x)$ under $S_t$. Let $X_0 = X \cap M_0$. Define $d^{(s)} = \dim_H(W_\theta^{(s)}(x) \cap M_0)$ and $d^{(u)} = \dim_H(W_\theta^{(u)}(x) \cap M_0)$. Then using Lemma \ref{packing dim lemma}, we get

$$d^{(u)} = \dim_H X_0 \in [\frac{-\ln(u-1)}{\ln \lambda}, \frac{-\ln(u-1)}{\ln \mu}].$$

We can use the same estimate for $d^{(s)}$, since $W^{(u)}_{\theta} = \mbox{Refl} W^{(s)}(\mbox{Refl}(x))$, where Refl: $\hat{Q} \rightarrow \hat{Q}$ is a bi-Lipschitz involution given by

\begin{displaymath}
   \mbox{Refl}(q,v) = \left\{
     \begin{array}{lr}
       (q, -v) & \mbox{ for } q \in \mbox{int} Q\\
       (q, 2 \langle n_K(q), v \rangle n_K(q) - v \rangle), & \mbox{ for } q \in \partial K.
     \end{array}
   \right.
\end{displaymath} 

If $E,F$ are Borel sets, the following inequalities are well known (see \cite{FractalGeometry}).

$$\dim_H E + \dim_H F \leq \dim_H(E \times F) \leq \dim_H E + \overline{\dim_p} F.$$

Lemma \ref{packing dim lemma} gives $\overline{\dim_p}(\Sigma^+, d_{\theta}) = \dim_H(\Sigma^+, d_\theta)$. Let $V$ be a neighbourhood of $M_0$ and let $U \subset V$ be a neighbourhood of $x$. Let $\varepsilon$ be small enough that $W_{\varepsilon}^{(u)}(x), W_{\varepsilon}^{(s)}(x) \subset U$, and let $h : W_{\varepsilon}^{(u)}(x) \times W_{\varepsilon}^{(s)}(x) \rightarrow R$ be the usual local product map, where $R$ is an open neighbourhood of $x$. This holonomy is at least H\"{o}lder continuous. Let $\alpha$ be the H\"{o}lder constant of $h$, then using basic properties of Hausdorff dimension \cite{FractalGeometry} we have
\begin{equation}
\label{No HS inequality}
\alpha (d^{(s)} + d^{(u)}) \leq \dim_H (R \cap M_0) \leq \alpha^{-1} (d^{(s)} + d^{(u)}).
\end{equation}

If $\alpha = 1$ we have 
\begin{equation} \label{HS equation}
\dim_H (R \cap M_0) = d^{(s)} + d^{(u)}.
\end{equation}

\begin{thm} 
\label{main theorem}
Let $\lambda_1 = \frac{1}{1 + d_{\max} g_{\max}}, \mu_1 = \frac{1}{1 + d_{\min} g_{\min}}$. Assume that $\alpha = 1$. Then

\begin{equation}
\label{good dimension estimate}
\frac{-2 \ln(u-1)}{\ln \lambda_1} \leq \dim_H M_0 \leq \frac{-2 \ln(u-1)}{\ln \mu_1}.
\end{equation}

\begin{proof}
For any $a < g_{\min}, b > g_{\max}$, letting $\lambda(b) = \frac{1}{1 + d_{\max} b}, \mu(a) = \frac{1}{1 + d_{\min} a}$ we have 
$$\dim_H M_0 = \dim_H (R \cap M_0) = d^{(s)} + d^{(u)} \in [\frac{-2 \ln(u-1)}{\ln \lambda(b)}, \frac{-2 \ln(u-1)}{\ln \mu(a)}].$$ 

Taking limits $a \rightarrow g_{\min}$ and $b \rightarrow g_{\max}$, we get the result.
\end{proof}

\end{thm}

\section{Dimension product structure}
\label{DPS}

In this section we discuss what is currently known about the holonomy $h$. The holonomy is always Lipshitz if the diffeomorphism $B$ is \textit{conformal} on both the stable and unstable manifolds (see \cite{Barreira1} and \S7 of \cite{Pesin}). This is the case for the billiard ball map $B$ in $\mathbb{R}^2$ but not in higher dimensions. To see this, suppose one of the obstacles is the unit sphere centered on the origin, and consider an unstable manifold containing the points $(0,0,10)$, $(\frac{1}{2},0,10)$, $(0,\frac{1}{2},10)$, each with a ray in a direction sufficiently close to $(0,0,-1)$ that the rays collide with the sphere. These points form a right angle, but their image under $B$ does not, so $B$ does not always preserve angles on unstable manifolds and is not conformal.

However Stoyanov in \cite{Non-integrability} showed that a class of billiards satisfy a pinching condition, which would imply the stable and unstable manifolds are $C^1$. In the notation of this paper, a billiard satisfies the pinching condition if $\lambda_0^{d_{\max}} < \mu_0^{2 d_{\min}}$, where $\lambda_0^{-1} = 1 + d_{\max}(\frac{1}{d_{\min}} + \frac{2 \kappa^+}{\cos \phi^+})$ and $\mu_0^{-1} = 1 + 2 d_{\min} \kappa^- \cos \phi^+$. In fact we will show that it holds when $\lambda(a)^{d_{\max}} < \mu(b)^{2 d_{\min}}$.

Hasselblatt and Schmeling in \cite{HSconjecture} proposed the conjecture that equation (\ref{HS equation}) holds generically or under mild hypotheses, even for non-conformal diffeomorphisms and non-Lipschitz holonomies. They proved this conjecture for a class of Smale solonoids. If the conjecture is shown to be true, at least in the case of dynamical billiards, then we recover the equation (\ref{good dimension estimate}). If not, then the result still holds for the class of billiards in \cite{Non-integrability}. We now calculate the constant $\alpha$ to get an estimate in terms of constants related to the billiard.

\section{Calculating the H\"{o}lder constant}
\label{Holder const}

A combination of arguments from \cite{Non-integrability, Hassleblatt} and Section \ref{better est} can be used to calculate the H\"{o}lder constant $\alpha$ for the holonomies. The open billiard flow $S_t$ is an example of an Axiom A flow, with hyperbolic splitting into $TM = E^{su} \oplus E^{ss} \oplus E^{S}$. These are the strong stable manifold, strong unstable manifold and the direction of the flow $S$ respectively. That is, for some $0 < \eta < 1$ we have $\|d S_t(u)\| \leq C \eta^t \|u\|$ for all $u \in E^s(t)$ and $t \geq 0$, and $\|d S_t(u)\| \leq C \eta^{-t} \|u\|$ for all $u \in E^u(t)$ and $t \leq 0$.

For each point $x$ there exist $\alpha_x < \beta_x < 0 < \alpha'_x < \beta_x$ such that for $v \in E^{ss}(x)$, $u \in E^{su}(x)$ and $t > 0$ we have
$$\frac{1}{C} e^{\alpha_x t} \|v\| \leq \|d S_t(x) \cdot u\| \leq C e^{\beta_x t} \|v\|, \mbox{ and}$$ 
$$\frac{1}{C} e^{-\alpha'_x t} \|u\| \leq \|d S_{-t}(x) \cdot u\| \leq C e^{-\beta'_x t} \|u\|.$$

In the case of billiards, the reflection property $W^{(u)}_{\theta} = \mbox{Refl} W^{(s)}(\mbox{Refl}(x))$ implies that $\alpha_x = -\alpha'_x$ and $\beta_x = -\beta'_x$. The H\"{o}lder constant $\alpha$ is then given by the \textit{bunching constant} $\alpha = B^u(S) = \displaystyle \inf_{x \in M_0} \frac{\beta_x - \beta'_x} {\alpha_x} = \inf_{x \in M_0} \frac{2 \beta_x}{\alpha_x}$ \cite{Hassleblatt}. The system is said to satisfy the \textit{pinching condition} if there exist $0 < \alpha_0 \leq \beta_0$ such that $0 \leq \alpha_0 \leq \alpha'_x \leq \beta'_x \leq \beta_0$ and $2 \alpha_x - \beta_x \geq \alpha_0$ for all $x \in M_0$.

Let $\hat{X} = S_{\tau}(W^{(u)}_{\theta}(x))$ for some small $\tau$, let $t > d_1(x) + \ldots + d_n(x)$ and let $\delta_j(s)$ be defined as in section \ref{deltature}. Then from \cite{Non-integrability}, there are constants $c_1, c_2$ such that
\begin{align*}
\frac{c_1}{c_2} \frac{\|u\|}{\delta_1(0) \delta_2(0) \ldots \delta_n(0)} &\leq \|d S_t(x) \cdot u\| \leq \frac{c_2}{c_1} \frac{\|u\|}{\delta_1(0) \delta_2(0) \ldots \delta_n(0)}\\
\frac{c_1}{c_2} \frac{\|u\|}{\mu_0^{n_0} \mu^{n-n_0}} &\leq \|d S_t(x) \cdot u\| \leq \frac{c_2}{c_1} \frac{\|u\|}{\lambda_0^{n_0} \lambda^{n-n_0}}.\\
\frac{c_1}{c_2} \left(\frac{\mu}{\mu_0} \right)^{n_0} \mu^{-t / d_{\max}} \|u\| &\leq \|d S_t(x) \cdot u\| \leq \frac{c_2}{c_1} \left (\frac{\lambda}{\lambda_0} \right)^{n_0} \lambda^{-t / d_{\min}} \|u\| \\
A e^{-t \ln \mu / d_{\max}} \|u\| &\leq \|d S_t(x) \cdot u\| \leq B e^{- t \ln \lambda / d_{\min}} \|u\|,
\end{align*}

\noindent where $\lambda = \lambda(b) = \frac{1}{1+ d_{\max} b}$, $\mu = \mu(a) = \frac{1}{1 + d_{\min} a}$, while $A = A(a,b)$ and $B = B(a,b)$ are new global constants that exist for all $a < g_{\min}, b > g_{\max}$ (these are not necessarily bounded above). This inequality holds for all $t \geq t_0$ with $t_0$ sufficiently large that $m > n_0$, but there must be constants $A'$ and $B'$ such that the same inequality holds for all $0 < t \leq t_0$. Taking $C$ large enough that $C > \max\{B,B'\}$ and $\frac{1}{C} < \min\{A,A'\}$, we now have $\alpha_x = -\ln \mu / d_{\max}$ and $\beta_x = -\ln \lambda/ d_{\min}$ so the bunching constant is $\displaystyle B^u(S) = \frac{2 d_{\min} \ln \mu}{d_{\max} \ln \lambda}$. This argument improves Proposition 1.2 in \cite{Non-integrability} by replacing $[\mu_0,\lambda_0]$ with the smaller interval $[\mu, \lambda]$ for any $a < g_{\min}, b > g_{\max}$.

\begin{prop}
Let $a < g_{\min}, b > g_{\max}$. Assume that $\lambda(b)^{d_{\max}} < \mu(a)^{2 d_{\min}}$ and the boundary $\partial K$ is $C^3$. Then the open billiard flow in the exterior of $K$ satisfies the pinching condition on its non-wandering set $M_0$. For any $x \in M_0$ we can choose $\alpha_x = \alpha_0 = \frac{\ln \mu(a)}{d_{\max}}$ and $\beta_x = \beta_0 = \frac{\ln \lambda(b)}{d_{\min}}$. 
\end{prop}

We cannot take the limit as $a \rightarrow g_{\min}, b \rightarrow g_{\max}$ for this proposition, since the constants $A$ and $B$ may not be bounded above. However when $\lambda^{d_{\max}} < \mu^{2 d_{\min}}$ we have $\alpha = 1$ so equations (\ref{HS equation}) and (\ref{good dimension estimate}) hold. Taking limits we can extend this to $\lambda_1^{d_{\max}} < \mu_1^{2 d_{\min}}$, which proves part 2 of the main theorem. If (\ref{HS equation}) does not hold then we have the following general estimate using (\ref{No HS inequality}):

\begin{equation}
-\frac{4 d_{\min} \ln \mu_1 \ln(u-1)}{d_{\max} (\ln \lambda_1)^2} \leq \dim_H M_0 \leq -\frac{d_{\max} \ln \lambda_1 \ln(u-1)}{d_{\min} (\ln \mu_1)^2}.
\end{equation}

\section{Improvement of estimates}
\label{better est}

\subsection{Convex hull conjecture}
\label{Convex hull conjecture}

We propose a conjecture that restricts the non-wandering set to a smaller area. This allows some relaxation of conditions. 

\begin{defn}
For any $i \neq j$, let $(p_{ij}, p_{ji}) \in K_i \times K_j$ denote the minimum of $F: K_i \times K_j \rightarrow \mathbb{R}, (q_1,q_2) \mapsto \|q_1 - q_2\|$. Then each $p_{ij}$ is on the boundary $\partial K_i$ and the vector $p_{ji} - p_{ij}$ is normal to $\partial K_i$ at $p_{ij}$. 
\end{defn}

\begin{conj}
Denote the convex hull Cvx$\{p_{ij}: 1 \leq i, j \leq n, i \neq j \}$ by $H$. Let $1 \leq \alpha_1, \ldots , \alpha_n \leq u$ $(n \geq 3)$ be a finite sequence of indices and let $(q_1, \ldots, q_n)$ be a periodic billiard trajectory such that $q_j \in K_{\alpha_k}$ for each $j$. Then each $q_j$ is contained in $H$. Furthermore, the non-wandering set $M_0$ is contained in $H$.

\end{conj}
We prove this conjecture for the case of an $3$-dimensional billiard in which the obstacles are spheres. A very similar proof will work for all two-dimensional billiards, and higher dimensional billiards with hyperspherical obstacles. The general case in higher dimensions may be more difficult.

\begin{proof}[Proof of the conjecture for spherical obstacles]
If the obstacles are spheres, then $H \cap Q$ is simply the convex hull of the centres of the spheres intersected with $Q$. Suppose that $(q_1, \ldots , q_n)$ is a periodic trajectory, but that at least one point is outside $H$. Without loss of generality we can number the points and obstacles such that $q_1 \notin H$ and $\alpha_1 = 1$. $H$ is bounded by a number of planes, so $q_1 \in K_{1}$ is on the outside (i.e. the side not containing $H$) of one such plane, say $\Pi = \Pi_{123}$, determined by the centres of obstacles $K_1, K_2, K_3$. Let $\nu$ be the outward normal vector of $\Pi$ and denote $v_j = \frac{q_{j+1} - q_j}{\|q_{j+1} - q_j\|}$, (with the convention that $q_0 = q_n$). Without loss of generality, assume that $v_0 \cdot \nu > 0$.

For each $k \geq 1$ we have $q_{k+1} = q_k + d_k v_k$ and $v_{k} = v_{k-1} - 2 \langle v_{k-1}, n_K(q_k) \rangle n_K(q_k)$. We also have $\langle v_{k-1}, n_K(q_k) \rangle < 0$. We show by induction that $q_k \cdot \nu > q_1 \cdot \nu$ and $v_{k-1} \cdot \nu > v_0 \cdot \nu$ for all $k > 1$. 

Suppose $q_k \in \partial K_{\alpha_k}$ is on the outside of $\Pi$ and $v_{k-1} \cdot \nu > v_0 \cdot \nu$. The centre of $\partial K_{\alpha_k}$ is on the inside of $\Pi$, so the normal vector $n(q_k)$ must point away from $\Pi$, i.e. $n_K(q_k) \cdot \nu > 0$. So $v_k \cdot \nu = v_{k-1} \cdot \nu - 2 \langle v_{k-1}, n_K(q_k) \rangle n_K(q_k) \cdot \nu > v_{k-1} \cdot \nu > v_0 \cdot \nu$. Then $q_{k+1} \cdot \nu = q_k \cdot \nu + d_k v_k \cdot \nu > q_k \cdot \nu$. So $q_{k+1}$ is also on the outside.

For the orbit to be periodic we must have $q_1 = q_{n+1}$ for some $n$. So by contradiction, all periodic points must be contained in $H$. Since $H$ is a closed set and the periodic points are dense in $M_0$, we have $M_0 \subset H$.


\end{proof}

\begin{cor}[Corollary 1]
Given a billiard for which the above conjecture is true, the non-wandering set $M_0$ is entirely contained in $H$, which means any change to the billiard outside of $H$ will not have any effect on the non-wandering set, unless it introduces a new periodic point. This means all results in this paper (and perhaps others) apply to billiards that are not smooth or convex, or that violate the no-eclipse condition $(\textbf{H})$, provided that the intersection $K \cap H$ still satisfies these conditions.
\end{cor}

\begin{cor}[Corollary 2]
\label{Corollary 2}
In cases where the conjecture is true, we can use the set $H$ to find better estimates for billiard constants. For example, we can estimate $d_{\max} \leq \emph{diam } H$. The minimum and maximum curvatures over $M_0$ can be estimated by $\kappa^- \leq \displaystyle \min_{q \in \partial K \cap H} \kappa_{\min}(q)$ and $\kappa^+ \leq \displaystyle \min_{q \in \partial K \cap H} \kappa_{\max}(q)$. 

\end{cor}

\subsection{Adjusted domain of $g$}
Recall that the natural domain for the function $g$ is $[\kappa^- \cos \phi^+, \frac{\kappa^+}{\cos \phi^+}] \times [d_{\min}, d_{\max}]$. This applies in billiards where the dimension $D > 2$; when $D = 2$ the natural domain is $[\kappa^-, \frac{\kappa^+}{\cos \phi^+}] \times [d_{\min}, d_{\max}]$ (see the end of section \ref{Estimating Theta}). To cover both cases at once, we let $\iota = 0$ if $D = 2$ and $\iota = 1$ if $D > 2$, so that $\cos^{\iota} \phi$ is $1$ if $D=2$ and $\cos \phi$ otherwise. Define the adjusted domain by $$\mathbb{D} = \displaystyle \bigcup_{i,j} [\kappa^-_i \cos^{\iota} \phi^+_{ij}, \frac{\kappa^+_i}{\cos \phi^+_{ij}}] \times [d^-_{ij}, d^+_{ij}]$$

where $\kappa^-_i, \kappa^+_i$ are the minimum and maximum curvatures on $\partial K_i \cap H$, $d^-_{ij} \geq |p_{ij} - p_{ji}|, d^+_{ij} \leq \displaystyle \max_{k,l} |p_{ik} - p_{jl}|$ are the minimum and maximum distances between $K_i \cap H$ and $K_j \cap H$, and $\phi_{ij} = \max \{\phi(x) : x \in K_i \cap H, B x \in K_j \cap H \}$ is the maximum collision angle over trajectories from $K_i$ to $K_j$. These $\phi_{ij}$ can be estimated by $\cos \phi_{ij} \geq \frac{b^-_{ij}}{d_{\max}}$ where $b^-_{ij} = \displaystyle \min_{k} d(K_i, Cvx(K_j, K_k))$.

The minimum and maximum values of $g$ over the natural domain may be outside of the adjusted domain. The minimum and maximum values in the adjusted domain are given by 

\begin{align*}
g_{\min} &= \min \{ g(\gamma, \theta) : (\gamma, \theta) \in \mathbb{D} \} = \min \{ g(\kappa^-_i \cos^{\iota} \phi_{ij}, d^+_{ij}), 1 \leq i,j, \leq u \}\\
g_{\max} &= \max \{ g(\gamma, \theta) : (\gamma, \theta) \in \mathbb{D} \} = \max \left\{ g \left(\frac{\kappa^+_j}{\cos \phi_{ij}}, d^-_{ij} \right), 1 \leq i,j, \leq u \right\}
\end{align*}

\begin{lem}
For any $x = (q,v) \in M_0$, we have $(\frac{\kappa(q_j)}{\cos \phi_j(x)}, d_j(x)) \in \mathbb{D}$ and $(\kappa(q_j) \cos^{\iota} \phi_j(x)), d_j(x)) \in \mathbb{D}$ for all $j \in \mathbb{Z}$. 
\begin{proof}
Assume $D > 2$. Since $q_j = \pi B^j x \in M_0$ for all $j \in \mathbb{Z}$, we have
$\kappa(q_j) \in [\kappa_{i_j}^-, \kappa_{i_j}^+], \phi(B^j x) \in [0, \phi^+_{i_j i_{j+1}}]$, and $d(q_j, q_{j+1}) \in [d^-_{i_j i_{j+1}}, d^+_{i_j i_{j+1}}]$. Hence there exist some integers $1 \leq a,b, \leq u$ such that $\kappa(q_j) \cos \phi(B^j x) \geq \kappa^-_a \cos \phi^+_{ab}$ and $\frac{\kappa(q_j)}{\cos \phi(B^j x)} \leq \frac{\kappa^+_a}{\cos \phi^+_{ab}}$. For the same $a,b$ we have $d(q_j) \in [d^-_{ab},d^+_{ab}]$. The proof for $D = 2$ is analogous.
\end{proof}

\end{lem}

\begin{exm}
Consider the billiard displayed in Figure 1 consisting of three disks arranged in an isoceles triangle of height $10$ and base length $8$. The disks $K_1, K_2, K_3$ have radii 1, 2 and 3 respectively. The solid lines give the distances $d_{ij}^-$ and the dashed lines give the distances $d_{ij}^+$. Figure 2 displays the adjusted domain over the natural domain, with contour lines of the function $g(\gamma, \theta)$. The following calculations were obtained using the programs Geogebra and Mathematica. 

Using the adjusted domain rather than the natural domain means that the interval $[g_{\min},g_{\max}]$ is reduced from $[0.760, 7.34]$ to $[0.762, 3.41]$. Using the natural domain we have the estimate $$0.326 \leq \dim_H M_0 \leq 1.167,$$ but with the adjusted domain we get $$0.396 \leq \dim_H M_0 \leq 1.165.$$

\includegraphics[width=70mm]{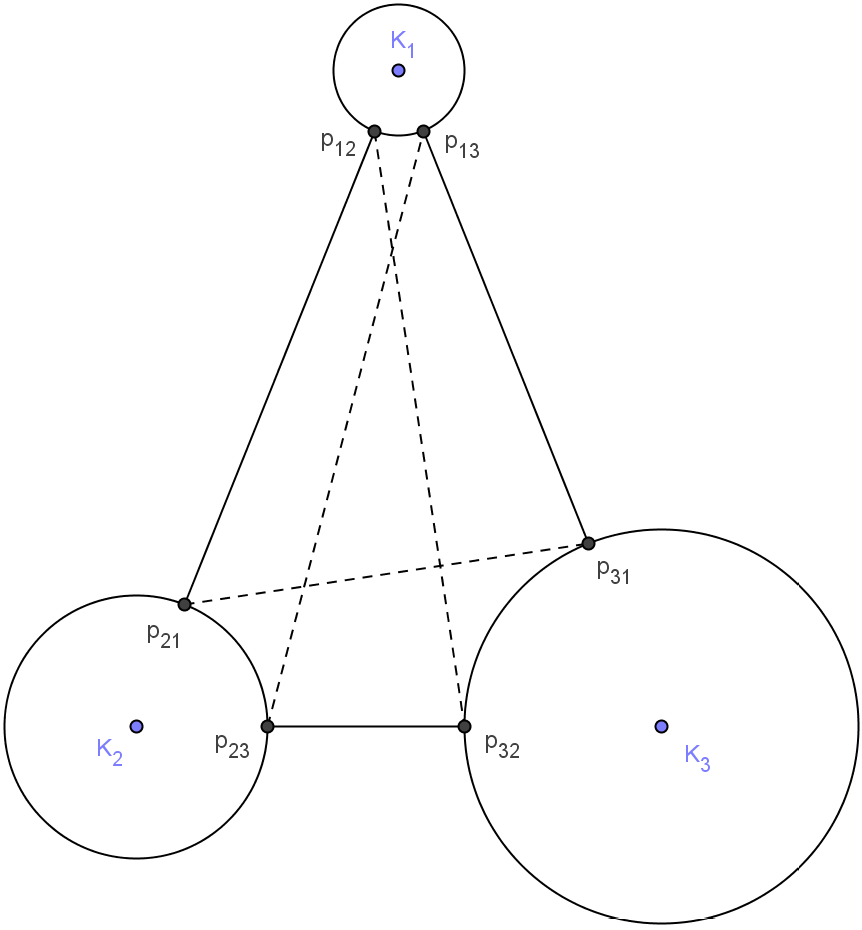}\\
\includegraphics[width=80mm]{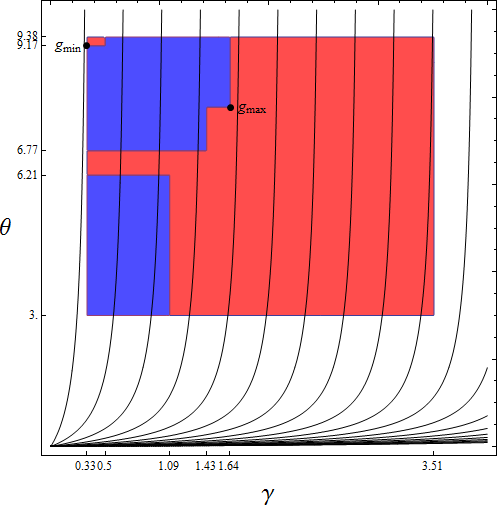}

\end{exm}

\end{document}